\documentclass{amsart}
\usepackage{ifthen}
\usepackage{verbatim}
\usepackage{amssymb,amsmath}
\usepackage{enumerate}
\usepackage[usenames,dvipsnames,svgnames,table]{xcolor}
\usepackage[colorlinks=true]{hyperref}

\usepackage{etoolbox}

\numberwithin{equation}{section}

\def\phi{\varphi}
\def\R{\mathbb R}
\def\N{\mathbb N}
\newcommand{\ch}{\mathcal{H}}
\newcommand{\cm}{\mathcal{M}}
\newcommand{\cd}{\mathcal{D}}

\newcommand{\eps}{\varepsilon}
\newcommand{\ra}{\rightarrow}
\newcommand{\be}{\begin{equation}}
\newcommand{\ee}{\end{equation}}

\newcommand{\p}{\partial}

\newcommand{\res}{\ensuremath{\,\textsf{\small \upshape L}\,}}

\newtheorem{theorem}{Theorem}[section]
\newtheorem{lemma}[theorem]{Lemma}
\newtheorem{corollary}[theorem]{Corollary}

\newtheorem{proposition}[theorem]{Proposition}
\theoremstyle{definition}
\newtheorem{definition}[theorem]{Definition}

\theoremstyle{remark}
\newtheorem{remark}[theorem]{Remark}
\numberwithin{equation}{section}
\def\uclhome{@ucl.ac.uk}
\def\stanfordhome{@stanford.edu}

\begin{document}

\title[Mean Curvature Flow  with triple Edges]{A local regularity theorem for Mean Curvature Flow  with triple Edges}
\author{Felix Schulze}
\address{Felix Schulze: 
  Department of Mathematics, University College London, 25 Gordon St,
  London WC1E 6BT, UK}
\curraddr{}
\email{f.schulze\uclhome}

\author{Brian White}
\thanks{\noindent The research of the first author was partially supported by NSF grant DMS--1440140 while the author was in residence at the Mathematical Sciences Research Institute in Berkeley, California during the Spring 2016 semester. The research of the second author was partially supported by NSF grant DMS--1404282.}
\address{Brian White: Department of Mathematics, Stanford University,
  Stanford, CA 94305, USA}
\curraddr{}
\email{bcwhite\stanfordhome}
\subjclass[2000]{}

\dedicatory{}


\begin{abstract} Mean curvature flow of clusters of
  $n$-dimensional surfaces in $\R^{n+k}$ that meet in triples at
  equal angles along smooth edges and higher order junctions on lower dimensional faces
  is a natural extension of classical mean curvature
  flow. We call such a flow a mean curvature flow with
  triple edges. We show that if a smooth mean curvature flow with
  triple edges is weakly close to a
  static union of three $n$-dimensional unit density half-planes, then
  it is smoothly close. Extending the regularity result to a class of
  integral Brakke flows, we show that this implies smooth short-time
  existence of the flow starting from an initial surface cluster that has
  triple edges, but no higher order junctions.
\end{abstract}

\maketitle

\section{Introduction}
We consider smooth families $\cm = \cup_{t\in I} M_t \times \{t\}$ of
$n$-dimensional surface clusters in $\R^{n+k}$. Such a cluster has the form
\[
    M_t=  \bigcup_{i=1}^N M^i_t,
\]
where  each $M^i_t$ is the image of a smooth family of embeddings $X^i_t:P^i \rightarrow \R^{n+k}$, where $P^i \subset \R^n$ is a bounded, open, convex polytope. We assume that $X^i_t$ extends to a smooth family of immersions into $\R^{n+k}$ of an open neighbourhood $U^i \subset  \R^n$ of $\overline{P}^i$. We assume further that the $M^i_t$ are disjoint away from their boundaries and that they meet in triples
at equal angles along their $(n-1)$-dimensional faces and along each $l$-dimensional face, for $0\leq l \leq n-2$, modelled on a stationary polyhedral cone with an $l$-dimensional translational symmetry. We call the $(n-1)$-faces where three sheets meet edges and lower dimensional faces where necessarily more than three sheets meet higher order junctions.

We say that $\cm = \cup_{t\in I} M_t \times \{t\}$ solves mean
curvature flow, if, given the parametrisation $X_t =\cup_{i=1}^n X^i_t$ of the moving
cluster, the velocity vector satisfies
\[
\Big(\frac{\partial}{\partial t}X\Big)^\perp = \vec{H}, 
\]
where ${}^\perp$ is the projection onto the normal space along each
sheet and $\vec{H}$ its mean curvature vector. Along the edges and higher order junctions, we
require that this holds for each sheet separately. 

We denote the backwards parabolic cylinder with
radius $r$, centred at a space-time point $X= (x,t) \in \R^{n+k}\times \R$, by
$$C_r(X) =  B_r(x)\times (t-r^2,t)\ .$$
We will write $O$ to denote the origin $(0,0)$ in space-time.

\begin{theorem}
  \label{theorem1}
Let $\cm^j$ be a sequence of smooth, $n$-dimensional mean curvature flows
with triple edges in $\R^{n+k}$ that converge as Brakke flows to
a static union of 3 unit density $n$-dimensional half-planes in $C_2(O)$. Then the convergence is smooth in $C_1(O)$.
\end{theorem}

We also consider the class of integral Brakke flows that are
$Y$-regular in the following sense: if $P$ is a space-time point with Gaussian density one, 
or with a tangent  flow consisting of a static union of 3 unit density half-planes,
then $P$ has a space-time neighbourhood in which the flow is smooth. The
above theorem remains true for $Y$-regular flows: see Theorem \ref{Y-regular-theorem}.
We also show that the class of $Y$-regular flows is closed under weak convergence: 
see Corollary \ref{Y-regular-corollary}.

Combining this with Ilmanen's elliptic regularisation scheme, we show:

\begin{theorem}\label{short-time-existence-theorem} Let $M_0$ be a smooth, compact $n$-surface cluster in
$\R^{n+k}$ without higher order junctions, i.e. a finite union of compact manifolds-with-boundary
that meet each other at 120 degree angles along their smooth
boundaries. Then there exists a $T>0$
and a smooth solution to mean curvature flow with triple junctions
$(M_t)_{0<t<T}$ such that $M_t \ra M_0$ in $C^1$, and in $C^\infty$ away
from the triple junctions. 
\end{theorem}
In codimension $k=1$, Theorem \ref{theorem1} implies that the solution exists until 
the supremum of the second fundamental form over the cluster blows up,
or until two triple edges collide: see Corollary
\ref{maximal-existence-corollary}.   For surface clusters with higher order junctions (under a topological restriction), we show that there exists a $Y$-regular Brakke flow starting from such a cluster, where the initial cluster is attained in $C^\infty$ away from the junctions and in $C^1$ at the triple junctions: see Theorem \ref{thm:existence_general}. The existence of a Brakke flow starting from such a cluster in codimension $k=1$ has also recently been established by Kim and Tonegawa \cite{KimTonegawa15}. For the flow of networks in arbitrary codimension, i.e.~ the case $n=1$, the assumption of equal angles initially is not necessary, as long as they are positive: see Theorem \ref{short-time-existence-nonregular-network-theorem}.

The corresponding fundamental regularity theorem for smooth mean
curvature flow was proven by the second author \cite{White05}. Mean
curvature flow with triple edges for curves in codimension one is the
network flow. A similar regularity theorem for smooth network flow was shown by Ilmanen and
Neves together with the first author, \cite[Theorem
1.3]{IlmanenNevesSchulze14}. For Brakke flows the fundamental
regularity theorem is due to Brakke, \cite{Brakke}. More recently,
Tonegawa and Wickramasekera have proven the analogous result for
1-dimensional integral Brakke flows close to a static union of three
half lines in the plane, \cite{TonegawaWickramasekera15}.

Smooth short time existence for the network flow was first established
by Mantegazza, Novaga and Tortorelli \cite{MantegazzaNetworks} using PDE
methods, following Bronsard and Reitich \cite{BronsardReitich93}. Short time existence of mean curvature flow with triple
edges, also in the PDE setting, was considered by Freire
\cite{ Freire10b, Freire10a} in
the case of graphical hypersurfaces and by Depner, Garcke and
Kohsaka in \cite{DepnerGarckeKohsaka14} for special hypersurface
clusters. Both the results of Freire and of Depner, Garcke and
Kohsaka require as well that no higher order junctions are present. Short time existence for the planar network flow, starting from an initial network with multiple points, where more than three curves are allowed to meet without a condition on the angles was shown by Ilmanen and
Neves together with the first author, \cite{IlmanenNevesSchulze14}.

{\bf Outline.} In section \ref{section-setup} we clarify the setup and recall some
notation.\\[0.5ex]
Let $\mathcal{Y}$ denote a flow consisting of three nonmoving
halfplanes that meet at equal angles along their common edge.
In section \ref{section-estimates}, we show that if a smooth mean curvature
flow with triple edges is bounded in $C^{2,\alpha}$ and is $C^2$ close to $\mathcal{Y}$
in a spacetime domain, then it is $C^{2,\alpha}$ close to $\mathcal{Y}$ 
in a smaller spacetime domain. We do this by writing the solution as a perturbation of
an approximating solution of the heat equation. We use standard
Schauder estimates for the heat equation to show first that
the perturbation decays in $C^{2,\alpha}$, and use this information,
together with the 120 degree condition along the triple edge, again
using only standard
Schauder estimates for the heat equation, to show
that also the approximating solutions of the heat equation converge in
$C^{2,\alpha}$. We use this, together with a blow up argument (analogous to the
one used in \cite{White05}) to prove Theorem \ref{theorem1}.\\[0.5ex]
In section \ref{section-extension} we extend  Theorem~\ref{theorem1}
to the class of $Y$-regular Brakke flows.
The main ingredient is showing that  a
static union of three unit density half-planes is, up to rotations,
weakly isolated in the space of self-similarly shrinking integral
Brakke flows. \\[0.5ex]
In section \ref{section-short-time} we prove
Theorem \ref{short-time-existence-theorem}, using flat chains mod $3$ in Ilmanen's elliptic
regularisation scheme to get long-time existence, and using the results
from section \ref{section-extension} to get short-time regularity.
The main ingredient is showing that the
Brakke flow constructed via elliptic regularisation has only unit
density static planes and static unions of three unit density half
spaces as tangent flows for some initial time interval.\\[0.5ex]
 In section \ref{section-general-initial-clusters} we prove Theorem \ref{thm:existence_general} by writing the initial cluster as a flat cycle with coefficients in the $(k-1)$-th homology group of the complement with coefficients in $\mathbb{Z}_2$ and adapting the elliptic regularisation scheme to this setting.\\[0.5ex]
In section \ref{section-nonregular-network} we prove Theorem \ref{short-time-existence-nonregular-network-theorem}, showing that for the network flow in arbitrary codimension the assumption of equal angles at the inital triple junctions is not necessary.

\section{Setup and Notation} \label{section-setup}
\noindent For the definition of Brakke flows and an overview of the
fundamental properties we refer the reader to
\cite{Ilmanen94}. It is important to note that the boundary condition
along the faces of the moving surface clusters implies that a smooth
mean curvature flow with triple edges constitutes a Brakke flow. 

Let $\cm^i$ be a sequence of Brakke flows in an open subset $U$ of
space-time. We say that $\cm^i$ converge weakly to a limiting Brakke flow
$\cm$ in $U$ if 
\begin{equation}\label{eq:weak conv}
\int_{-\infty}^\infty \int_{\R^{n+k}} \varphi \, d\mu^i_t\, dt \ra
\int_{-\infty}^\infty \int_{\R^{n+k}} \varphi \, d\mu_t\, dt 
\end{equation}
for any $\varphi \in C^1_c(U;\R)$, where $\{\mu^i_t\}$ and $\{\mu_t\}$ are the
families of Radon measures corresponding to $\cm^i$ and $\cm$,
respectively. This is equivalent to  $\mu^i_t$ converging weakly to  $\mu_t$
for all but a countable set of $t$.  (Of course, by passing to a subsequence,
one can assume that the $\mu^i_t$ converge weakly for all $t$, though
\eqref{eq:weak conv} does not imply that the limit equals  $\mu_t$ for all $t$.) 
The compactness theorem for integral Brakke flows of
Ilmanen, \cite[Theorem 8.1]{Ilmanen94}, guarantees that any sequence
of integral Brakke flows on an open subset $U$ of space-time, with
locally uniformly bounded mass, has a subsequence converging weakly to a
limiting integral Brakke flow.

We denote parabolic rescaling of space-time with a factor $\lambda >0$
by $\cd_\lambda$, that is
$$\cd_\lambda(X)=\cd_\lambda(x,t) = (\lambda x, \lambda^2t)\, .$$
Note that if $\cm$ is a Brakke flow in $U$, then $\cd_\lambda\cm$ is a
Brakke flow in
$\cd_\lambda U$.

Let the $n$-dimensional half-space $H \subset \R^{n+k}$ be given by
$$H:=\{x\in \R^{n+k}\, |\, x_{n+1}= x_{n+2} = \ldots = x_{n+k} = 0\, ,\  x_1\geq 0\}$$
and let $S$ be the counterclockwise rotation in the $x_1x_{n+1}$-plane by
$2\pi/3$ around $0$, extended to be a rotation of $\R^{n+k}$. The union
$$Y:=\cup_{i=1}^3S^{i-1}H$$
is a static configuration of three half-planes meeting along the
subspace 
$$L = \{x\in \R^{n+1}\, |\,  x_1= x_{n+1}= x_{n+2} = \ldots = x_{n+k}= 0\}\ .$$ 
We see that $Y$ is a static solution to mean curvature flow with
triple edges. We denote its space-time track by $\mathcal{Y}$. 

In the following, we will use parabolic H\"older norms and the associated spaces  $C^{q,\alpha}$. If clear from the context, we will not explicitly mention that we use parabolic norms. For more details on parabolic H\"older norms see \S 7 in \cite{White05}.

\section{Estimates in the smooth case}\label{section-estimates}

Consider a smooth flow $\cm$ with triple edges in $C_4(O)$. 

\begin{definition}\label{def:closeness} We say that $\cm$ is 
$\varepsilon$-close in $C^{2,\alpha}$ to $\mathcal{Y}$ in $C_4(O)$ if 
 we can decompose $M_t$ in $C_4(O)$ into three sheets 
$$M_t\cap B_4(0) =\cup_{i=1}^3M^i_t =\cup_{i=1}^3S^{i-1}T^i_t, $$
(where $M^i_t=S^{i-1}T^i_t$),
such that 
there exist maps 
\[
    F^i(\cdot,t): H\cap B_4(0) \to \R^{n+k}\, ,
\]
parametrising $T^i_t$ for $-16 < t < 0$, with the following properties:
\begin{enumerate}
\item $F^1$, $SF^2$, and $S^2F^3$ agree along the common edge: 
\[
 \text{$F^1(\cdot,t) = SF^2(\cdot,t)=S^2F^3(\cdot,t)$ on  $L\cap B_4(0)$ for $t\in (-16,0)$} ,
\]
\item\label{item:restriction}
 the restriction of the parametrisation to $L\cap B_3(0)$ is perpendicular to $L$:
\begin{equation*}
\pi_L\big(F^i(x,t) - x\big) = 0 
\text{ for all $x \in L\cap B_3(0),\ i=1,2,3$ and $t\in (-9,0)$} ,
\end{equation*}
where $\pi_L$ is the orthogonal projection onto $L$, and 
\item\label{item:inequality}
$$\|F^i-{id}_H\|_{C^{2,\alpha}((H\cap B_4(0))\times (-16,0)} \leq \varepsilon.$$
\end{enumerate}
Similarly, we say that 
$\cm$ is 
$\varepsilon$-close in $C^{2}$ to $\mathcal{Y}$ in $C_4(O)$
provided the above statements hold with $C^{2}$ in place of $C^{2,\alpha}$. In case that we make a statement about $C^{2}$ and $C^{2,\alpha}$ at the same time, then we assume that this is with respect to the same parametrisation $F^i$. 
\end{definition}

 Note that condition~\eqref{item:restriction} in the definition implies that the triple edge of $\cm$ is written as a graph over the triple edge of $\mathcal{Y}$ in $C_3(O)$. 

Let us assume that we have a flow $\cm$ that is $\varepsilon$-close to $\mathcal{Y}$ in $C^2$ and $\delta$-close in $C^{2,\alpha}$ in $C_4(O)$ for sufficiently small $\varepsilon, \delta >0$. We aim to use the given parametrisations $F^i$ to construct a new parametrisation in $C_3(O)$. We take $X^i: (H\cap B_3(0))\times [-9,0) \ra \R^{n+k}$ to be the solution to the linear heat equation that has the same initial values and boundary values as $F^i$ restricted to $H\cap C_3(O)$:
\begin{align*}
&\frac{\partial}{\partial t}X^i = \Delta^\delta X^i\, ,\\
&X^i(\cdot, -9) =F^i(\cdot, -9)\ ,\\
&X^i(\cdot, t)|_{\partial
  B_3(0)\cup(L\cap B_3(0))} = F^i(\cdot, t)\ ,
\end{align*}
where $\Delta^\delta$ is the Laplacian with respect to the Euclidean metric on $H$. Note that the domain $H\cap B_3(0)$ has corners. Nevertheless away from the corners of the domain, we can assume that $X^i$ is a $C^{2,\alpha}$-diffeomorphism 
onto its image, and that $X^i$ is  $C^2$-close to the identity mapping. 
We parametrise $F^i(H\cap B_3(0),t) \subset T^i_t$ by
$$ G^i(\cdot,t) =X^i(\cdot,t)+ u^i(\cdot, t) = X^i(\cdot,t)+
\sum_{l=1}^{k}u^i_l(\cdot, t)\, e_{n+l} $$
where $u^i:(H\cap B_3(0))\times [-9,0)\ra \{0\}^{n}\times \R^k$. Note that since both $F^i_t$ and $X^i_t$ are $C^2$ close to $id_H$, the maps $u^i$ are uniquely determined by requiring that they map to $\{0\}^{n}\times \R^k$, i.e. we write $T^i_t$ as a graph in $e_{n+1},\ldots, e_{n+k}$ direction over $\text{Im}(X^i_t)$. We can again assume that away from the corners of the domian, $u^i$ is
bounded in $C^{2,\alpha}$ and close in $C^2$ to zero, with boundary
values equal to zero. Similarly $G^i$ is bounded in $C^{2,\alpha}$ and close in $C^2$ to $id_H$, away from the corners of the domain.
\begin{proposition}\label{mainprop}
Let $\mathcal{M}^j$ be a sequence of smooth mean curvature flows with
triple edges in $C_4(O)$ that are $\delta$-close in $C^{2,\alpha}$ to $\mathcal{Y}$, for sufficiently small $\delta>0$, and
that converge in $C^2$ to $\mathcal{Y}$. Then they converge on $C_1(O)$ in
$C^{2,\alpha}$, i.e. in the above parametrisation
$$ \|u^{i,j}\|_{C^{2,\alpha}((H\cap B_1(0))\times [-1,0))} \ra 0\  \text{and}\ \|X^{i,j} -
  \text{id}_{H}\|_{C^{2,\alpha}((H\cap B_1(0))\times [-1,0))} \ra 0\, ,$$
as $j\ra \infty$ for $i=1,2,3$.
\end{proposition}
\begin{proof}
We consider $T^{i,j}_t$ and drop the indices
$i,j$ for the moment. We can assume that the tangent spaces to $T_t$ are
uniformly close to $\R^n\times \{0\}^{k}$. Since $T_t$ moves by mean
curvature flow we have
$$ \Big\langle \Big(\frac{\partial}{\partial t}  G_t\Big)^\perp,
e_{n+l}\Big\rangle = \big\langle \vec{H},e_{n+l}\big\rangle = \Delta x_{n+l}$$
where $l =1, \ldots, k$ and $\Delta$ is the Laplace-Beltrami operator on $T_t$. This is
equivalent to 
$$\Big\langle \frac{\partial}{\partial t}X + \frac{\partial}{\partial
  t}u, \pi^\perp(e_{n+l})\Big\rangle =
\Delta x_{n+l}\ ,$$
where $\pi^\perp$ is the orthogonal projection to the normal space of
$T_t$. Letting $\pi^*$ be orthogonal projection onto $\{0\}^n\times
\R^{k}$, we can write this as
\begin{equation*}
  \begin{split}
    \frac{\partial}{\partial t} X_{n+l} + \frac{\partial}{\partial
  t}u_{l} &= \Delta x_{n+l} + \Big\langle \frac{\partial}{\partial t}X + \frac{\partial}{\partial
  t}u, (\pi^*- \pi^\perp)(e_{n+l})\Big\rangle
\\
&=: \Delta x_{n+l}+E_{1,l}\ .
  \end{split}
\end{equation*}
Let $g_{ij}$ be the metric induced by $G_t$. We have
\begin{equation*}
  \begin{split}
    \Delta x_{n+l} &= \frac{1}{\sqrt{\det(g)}} \frac{\p}{\p
      x_i}\Big( \sqrt{\det(g)}g^{ij} \frac{\p}{\p
      x_j}(X_{n+l}+u_l)\Big)\\
    &=  g^{ij}\frac{\p^2}{\p x_i\p x_j}(X_{n+l}+u_l) +  \frac{1}{\sqrt{\det(g)}} \frac{\p}{\p
      x_i}\big( \sqrt{\det(g)}g^{ij}\big) \frac{\p}{\p
      x_j}(X_{n+l}+u_l)\\
    &= \Delta^\delta(X_{n+l}+u_l) + (g^{ij}-\delta^{ij})\frac{\p^2}{\p
      x_i\p x_j}(X_{n+l}+u_l)\\
    &\ \  + \frac{1}{\sqrt{\det(g)}} \frac{\p}{\p
      x_i}\big( \sqrt{\det(g)}g^{ij}\big) \frac{\p}{\p
      x_j}(X_{n+l}+u_l)\\
    &=: \Delta^\delta(X_{n+l}+u_l) + E_{2,l},
  \end{split}
\end{equation*}
where $\Delta^\delta$ is the standard laplacian on $H$. This yields
$$ \frac{\partial}{\partial t} X_{n+l} + \frac{\partial}{\partial
  t}u_l =  \Delta^\delta(X_{n+l}+u_l)+ E_{1,l} + E_{2,l}\ ,$$
and thus
\begin{equation}\label{eq:est.1}
\frac{\partial}{\partial t}u_l - \Delta^\delta u_l = E_{1,l} + E_{2,l}\ ,
\end{equation}
for $l= 1,\ldots, k$. 
Note we can estimate, where $[\cdot]_\alpha$ and $\|\cdot\|_0$
are the parabolic H\"older half-norm and the sup-norm on $C_{5/2}(O)$,
that

 $$   [E_{1,l}]_\alpha \leq \big[\tfrac{\partial}{\partial t}X + \tfrac{\partial}{\partial
  t}u\big]_\alpha \big\|\pi^*-\pi^\perp \big \|_0
    + \big\|\tfrac{\partial}{\partial t}X + \tfrac{\partial}{\partial
  t}u\big\|_0 \big[\pi^*-\pi^\perp\big]_\alpha $$

 and
\begin{equation*}
  \begin{split}
   [E_{2,l}]_\alpha &\leq \big[g^{ij}-\delta^{ij}\big]_\alpha \big\|\tfrac{\p^2}{\p
     x_i\p x_j}(X_{n+l}+u_l)\big\|_0
    + \big\|g^{ij}-\delta^{ij}\big\|_0 \big[\tfrac{\p^2}{\p
     x_i\p x_j}(X_{n+l}+u_l)\big]_\alpha \\
    &\ \ \ \ + \Big[\tfrac{1}{\sqrt{\det(g)}} \tfrac{\p}{\p
      x_i}\big( \sqrt{\det(g)}g^{ij}\big)\Big]_\alpha \big\|\tfrac{\p}{\p
      x_j}(X_{n+l}+u_l)\big\|_0 \\
&\ \ \ \ + \Big\|\tfrac{1}{\sqrt{\det(g)}} \tfrac{\p}{\p
      x_i}\big( \sqrt{\det(g)}g^{ij}\big)\Big\|_0\big[\tfrac{\p}{\p
      x_j}(X_{n+l}+u_l)\big]_\alpha\ .
  \end{split}
\end{equation*}
Similarly
\begin{equation*}
      \|E_{1,l}\|_0 \leq \big\|\tfrac{\partial}{\partial t}X + \tfrac{\partial}{\partial
  t}u\big\|_0 \big\|\pi^*-\pi^\perp \big \|_0
\end{equation*}
and
\begin{equation*}
 \begin{split}
\|E_{2,l}\|_0 &\leq \big\|g^{ij}-\delta^{ij}\big\|_0 \big\|\tfrac{\p^2}{\p
     x_i\p x_j}(X_{n+l}+u_l)\big\|_0\\    
    &\ \ \ \ + \Big\|\tfrac{1}{\sqrt{\det(g)}} \tfrac{\p}{\p
      x_i}\big( \sqrt{\det(g)}g^{ij}\big)\Big\|_0 \big\|\tfrac{\p}{\p
      x_j}(X_{n+l}+u_l)\big\|_0 \ . 
 \end{split}
\end{equation*}
Reintroducing the suppressed indices $i,j$, this implies that 
$$\|E_{1,l}^{i,j}\|_{0,\alpha}, \|E_{2,l}^{i,j}\|_{0,\alpha} \ra 0$$
on  $C_{5/2}(O)$ as $j\ra \infty$, for $i=1,2,3$ and $l=1,\ldots,k$.  Since $u^{i,j}$ converges to 
zero in $C^2$ on $C_{5/2}(O)$, 
this implies by \eqref{eq:est.1} and parabolic Schauder estimates for
the heat equation that 
\begin{equation}\label{eq:u-good}
     \|u^{i,j}\|_{2,\alpha} \ra 0
\end{equation}
as $j\ra \infty$ on $C_2(O)$. 

As above, we denote with $X^i_m$ the $m$-th coordinate function of $X^i$. 
By~\eqref{item:restriction} in Definition \ref{def:closeness}, we have that
\begin{equation*}
   X^i_m(x,t)= x_m
\end{equation*}
for $(x,t) \in (L\cap B_3(0))\times [-9,0)$ and $m=2,\dots,n$.
Standard 
Schauder estimates then imply that 
\begin{equation}
X^i_m(x,t) \ra x_m  \quad  \text{on}\ C_2(O)\cap H\ \text{in}\ C^{2,\alpha}
\end{equation} 
for $m=2,\ldots,n$.

Since $0= \sum_{i=1}^3S^{i-1}(e_1) = \sum_{i=1}^3S^{i-1}(e_{n+1})$ we
have
$$ 0= \sum_{i=1}^3\langle S^{i-1} F^i(x,t),S^{i-1}(e_1)\rangle =
\sum_{i=1}^3\langle F^i(x,t),e_1\rangle =
\sum_{i=1}^3X^i_1(x,t)$$
for $(x,t) \in (L\cap B_3(0))\times [-9,0)$. Similarly
$$\sum_{i=1}^3X^i_{n+1}(x,t) = 0$$
for $(x,t) \in (L\cap B_3(0))\times [-9,0)$. Let us define on
$C_3(O)\cap H$, 
$$h = \sum_{i=1}^3X^i_1\ \ \ \text{and}\ \ \ v =
\sum_{i=1}^3X^i_{n+1}\ .$$
Both functions solve the heat equation on $C_3(O)\cap H$ with zero
boundary values on $C_3(O)\cap \partial H$. Again standard Schauder
estimates for the heat equation imply that $h,v$ are uniformly bounded
in $C^{3,\alpha}$ on $C_2(O)\cap H$ and also
\begin{equation}\label{eq:sumconvergence}
h,v \ra 0 \ \ \text{on}\ C_2(O)\cap H\ \text{in}\ C^{2,\alpha}\
.
\end{equation}
For $l \in \{2, \ldots ,k \}, \ i,j \in \{1,2,3\}, \ i\neq j$
and $(x,t) \in (L\cap B_3(0)) \times [-9,0)$ we have
$$ 0 = F^i_{n+l}(x,t)-F^j_{n+l}(x,t) =  X^i_{n+l}(x,t) -
X^j_{n+l}(x,t) $$
and thus, as above, 
\begin{equation}\label{eq:boundaryhighercodim}
X^i_{n+l} - X^j_{n+l} \ra 0\ \  \text{on}\ C_2(O)\cap H\ \text{in}\ C^{2,\alpha}\
.
\end{equation}

\newcommand{\RR}{\mathbf{R}}

For $x\in H$ let us write
\begin{equation}\label{eq:defpertubation}
G^{i,j}(x,t) = x + \psi^{i,j}(x,t)\ .
\end{equation}
Our assumptions then imply that $\psi^{i,j}$ is bounded in $C^{2,\alpha}$
on $C_4(O)\cap H$ and converges to zero in $C^2$ as $j\ra \infty$. By the previous
estimates we can assume that 
\begin{equation*}
\|\psi^{i,j}_m\|_{C^{2,\alpha}(C_2(O)\cap H)} \ra 0
\qquad
\text{ for } m=2,\ldots, n.
\end{equation*}
 We suppress the index $j$ of the sequence for the rest of the argument.
\newcommand{\pdf}[2]{\frac{\partial #1}{\partial #2}}
 
For $i=1,2,3$, let $N^i$ denote the unit conormal to $T^i$ along the triple edge.  That is, $N^i$ is the unit vector with the following three properties:
it is a linear
combination of the vectors
\[
     \pdf{G^i}{x_m} = e_m + \pdf{\psi^i}{x_m}\qquad(m=1,\dots,n),
\]
it is perpendicular to 
\[
   \pdf{G^i}{x_m} \quad (m=2,\dots,n),
\]
and its inner product with $e_1$ is negative.  Note that $N^i=N^i(\pdf{\psi}{x_1},\dots,\pdf{\psi}{x_n})$ 
is a real analytic function of the $\pdf{\psi^i}{x_j}$.

Expanding $N^i$ as a power series, we have
\begin{align*}
N^i
&=
-e_1  - \sum_{m>n} \bigg\langle \pdf{\psi^i}{x_1}, e_m\bigg\rangle e_m
 +
\sum_{m=2}^n \bigg\langle \pdf{\psi^i}{x_m}, e_1\bigg\rangle e_m
+ 
Q
\end{align*}
where $Q$ denotes terms that are of degree $\ge 2$.  
Since the $\pdf{\psi}{x_j}$ are bounded in $C^{1,\alpha}$ and tend to $0$ in $C^1$, it follows
that $Q$ tends to $0$ in $C^{1,\alpha}$.  Thus
\begin{align}\label{eq:conormal}
N^i
&=
-e_1  - \sum_{m>n} \bigg\langle \pdf{\psi^i}{x_1}, e_m\bigg\rangle e_m
+
\sum_{m=2}^n \bigg\langle \pdf{\psi^i}{x_m}, e_1\bigg\rangle e_m
+ 
E
\end{align}
where $E$ denotes a term that tends to $0$ in $C^{1,\alpha}$.

Now $\sum_{i=1}^3S^{i-1}N^i=0$ since the three sheets meet at equal angles.
Also, 
\[
\sum_{i=1}^3S^{i-1} e_1=0
\]
and $S e_j=e_j$ for all $j$ other than $1$ and $n+1$.

Thus applying $S^i$ to \eqref{eq:conormal} and 
 summing from $i=1$ to $3$ gives
\begin{equation}\label{eq:conormal.1}
\begin{split}
\sum_{i=1}^3 
  \bigg(
   \bigg\langle e_{n+1}, \pdf{\psi^i}{x_1} \bigg\rangle S^{i-1} e_{n+1}
  + 
  \sum_{l=2}^k  &\bigg\langle e_{n+l}, \pdf{\psi^i}{x_1}  \bigg\rangle e_{n+l} \\
  &- 
  \sum_{m=2}^n \bigg\langle e_1, \pdf{\psi^i}{x_m}  \bigg\rangle e_m  \bigg)
  = E.
 \end{split}
\end{equation}

Now take the component of both sides of \eqref{eq:conormal.1} in the $e_{1}$ direction to get:
\[
\frac{\sqrt3}2 \left(  \frac{\partial \psi^2_{n+1}}{\partial x_1} - \frac{\partial \psi^3_{n+1}}{\partial x_1} \right)
=
 E.
\]

More generally, the same argument shows
\begin{equation}
  \label{eq:bdest1}
  \frac{\partial \psi^i_{n+1}}{\partial x_1} - \frac{\partial
    \psi^j_{n+1}}{\partial x_1} = E
    \qquad (i,j\in \{1,2,3\}).
\end{equation} 

For $l\ge 2$, taking the component of both sides of \eqref{eq:conormal.1} in the $e_{n+l}$ direction gives
\begin{equation}
  \label{eq:bdest2}
  \sum_{i=1}^3\frac{\partial \psi^i_{n+l}}{\partial x_1} = E  \qquad ( l=2,\dots, k)
\end{equation}
 on $C_2(O)\cap \partial H$. 

Recall that we have 
$$\psi^i_{n+1} = G^i_{n+1} = X^i_{n+1}+u^i_{n+1} .$$
So \eqref{eq:u-good} and \eqref{eq:bdest1} imply
\begin{equation}
  \label{eq:bdest3}
  \frac{\partial X^i_{n+1}}{\partial x_1} - \frac{\partial
    X^j_{n+1}}{\partial x_1} = E
\end{equation}
for $i\neq j$ on $C_2(O)\cap \partial H$. This implies that the
functions $w^{ij}:= X^i_{n+1}- X^j_{n+1}$ for $i<j$ are solutions to
the heat equation on $C_2(O)\cap H$ with Neumann boundary conditions
$$\frac{\partial w^{ij}}{\partial x_1} = E_{ij}$$
on $C_2(O)\cap\partial H$ with $\|E_{ij}\|_{C^{1,\alpha}}\ra 0$. But
then standard Schauder estimates for the heat equation on a
half-space, see for example Theorem 4 in \cite{Simon97}, imply that
\begin{equation}\label{eq:diffconvergence}\text
{$\|w^{ij}\|_{C^{2,\alpha}}\ra 0$ on $C_{3/2}(O)\cap H$.} 
\end{equation}

Now
\begin{align*}
   3X^1_{n+1} 
   &=   
   (X^1_{n+1} + X^2_{n+1} + X^3_{n+1}) 
   + 
   (X^1_{n+1} - X^2_{n+1}) 
   +
   (X^1_{n+1} - X^3_{n+1}) 
   \\
   &=
   v + w^{12} + w^{13}.
\end{align*}
Thus by \eqref{eq:sumconvergence} and \eqref{eq:diffconvergence}, $X^1_{n+1}$ tends to $0$ in $C^{2,\alpha}$ on $C_{3/2}(O)$.
Similarly $X^2_{n+1}$ and $X^3_{n+1}$ tend to $0$ in $C^{2,\alpha}$ on $C_{3/2}(O)$. Note that the maps $$S^{i-1}(X^i_1(x,t),X^i_{n+1}(x,t))$$ map to the same point in the $x_1x_{n+1}$-plane for $(x,t) \in L\cap B_2(0)\times (-4,0)$. Thus the coordinates $X^1_{n+1}(x,t), X^2_{n+1}(x,t), X^3_{n+1}(x,t)$ determine $X^1_{1}(x,t), X^2_{1}(x,t), X^3_{1}(x,t)$ for $(x,t) \in L\cap B_2(0)\times (-4,0)$  and we obtain
$$\|X^i_1\|_{C^{2,\alpha}}\ra 0\qquad \text{on}\ C_{3/2}(O)\cap \partial H$$  and so 
$$\|X^i_1\|_{C^{2,\alpha}}\ra 0\qquad\text{on}\ C_1(O)\cap H. $$

Similarly, from \eqref{eq:bdest2} we obtain that
$$\Big \|\sum_{i=1}^3 X^i_{n+l}\Big\|_{C^{2,\alpha}} \ra 0 \qquad \text{on}\ C_1(O)\cap H
$$
which implies together with \eqref{eq:boundaryhighercodim} that
 $$\|X^i_{n+l}\|_{C^{2,\alpha}}\ra 0\qquad \text{on}\ C_1(O)\cap H$$
for $i=1,2,3$ and $l=2,\ldots, k$.
\end{proof}

By interpolation, this yields the following corollary.

\begin{corollary}\label{lem:maincor}
  Let $\cm^j$ be a sequence of smooth mean curvature flows with triple
  edges.
  Suppose that the $\cm^j$ are sufficiently close in $C^{2,\alpha}$ to the static
  solution $\mathcal{Y}$ on $C_4(O)$ and that the $\cm^j$ converge as Brakke flows to
  $\mathcal{Y}$. Then the $\cm^j$ converge in $C^{2,\alpha}$ to $\mathcal{Y}$ on
  $C_1(O)$. 
\end{corollary}

We want to define the $C^{2,\alpha}$-norm of the triple edge,
including a control of a neighbourhood of the triple edge.  We will use the $K_{2,\alpha}$-norm, as defined in \cite{White05}. 

 \begin{definition}[$\tilde{K}_{2,\alpha}$-norm] \label{def:k2alphanorm} Let $\cm = \bigcup M_t \times\{t\}$
  be a smooth mean curvature flow with triple edges in an open subset $U$ of
space-time. Let $X=(x_0,t_0)\in U$ be a spacetime point on a triple edge of $\cm$. Suppose
first that $X=O$ and $C_1(O)\subset U$, such that the following holds:
\begin{enumerate}
	\item $\cm\cap C_1(O)$ contains one triple edge, and the $K_{2,\alpha}$-norm of the triple edge at $O$ is less or equal than one. Note that the triple edge is a smooth submanifold of $\R^{n+k}$ and thus its $K_{2,\alpha}$-norm is well-defined.\\[-2ex]
	\item  $M_t\cap B_1(0)$ consists of three sheets for all $t\in (0,1)$, diffeomorphic to $H \cap B_1(0)$, which meet at the triple edge. Furthermore, assume that
	$$\sup_{\cm'\cap C_1(O)} |A| \leq 1\, ,$$
\end{enumerate}
where $|A|$ is the norm of the second fundamental form on the sheets. We then say 
$$\tilde{K}_{2,\alpha}(\cm,X)=\tilde{K}_{2,\alpha}(\cm',O)\leq 1 .$$
Otherwise $\tilde{K}_{2,\alpha}(\cm,X)>1$. 

More generally, we let
$$\tilde{K}_{2,\alpha}(\cm,O) = \inf \{\lambda>0\, :\,
\tilde{K}_{2,\alpha}(\cd_\lambda\cm, O)\leq1\}. $$
Note that this includes the possibility for any $\lambda<\tilde{K}_{2,\alpha}(\cm,X)$
another triple edge appears in $D_\lambda(\cm)\cap C_1(O)$ and neither the
$K_{2,\alpha}$-norm of the triple edge or the supremum of $|A|$
approach $1$ as $\lambda \searrow \tilde{K}_{2,\alpha}(\cm,X)$.
Finally, if $X$ is any point on a triple edge of $M$ we let
$$\tilde{K}_{2,\alpha}(\cm,X)= \tilde{K}_{2,\alpha}(\cm-X,O)\ .$$
We have a similar Arzela-Ascoli Theorem as in
\cite[Theorem 2.6]{White05} and 
$\tilde{K}_{2,\alpha}(\cm,\cdot)$ scales like the reciprocal of
distance. Furthermore we define a norm on an open subset $U$ of
space-time by defining
$$\tilde{K}_{2,\alpha;U}(\cm)=\sup_{X\in \cm\cap U}d(X,U) \cdot
\tilde{K}_{2,\alpha}(\cm,X)$$
where $d(X,U)$ is the parabolic distance of $X$ to the parabolic boundary of
$U$. 
\end{definition}

\begin{remark}\label{rem:k2alphanorm} Assume that $\cm$ is a smooth mean curvature flow in $C_1(O)$ with one triple edge, passing through $O$, which separates $\cm$ into three evolving sheets. Assume further that the $\tilde{K}_{2,\alpha}$-norm of each point on the triple edge in $C_{1-\delta}(O)$ is bounded by $\delta >0$. Furthermore, assume that the second fundamental form of each sheet is bounded by $\delta$ as well. For $\delta$ sufficiently small, Schauder estimates imply that the parabolic $C^{2,\alpha}$-norm of each sheet, written as a graph over a suitable domain in $\R^n$, is bounded by a constant $\delta'$ on $C_{3/4}(O)$ where $\delta'\rightarrow 0$ as $\delta \rightarrow 0$. Thus one can construct a parametrisation of the three sheets as in Definition \ref{def:closeness}, such that $\cm$ is $\delta''$-close to $\mathcal{Y}$ in $C_{1/2}(O)$. Again we can assume that $\delta''\rightarrow 0$ as $\delta \rightarrow 0$.
\end{remark}

\begin{proof}[Proof of Theorem \ref{theorem1}] Recall
that we have a sequence $\cm^j$ of smooth mean curvature flows with
triple edges in $C_2(O)$, converging as Brakke flows to
the static solution $\mathcal{Y}$. 

We first note that by the local regularity theorem in \cite{White05}
we have smooth convergence away from the triple edge of
$\mathcal{Y}$. Furthermore, by the upper semicontinuity of the Gaussian density,
we can assume that there are no higher order junctions present in
$C_{3/2}(0)$\footnote{It is easy to see that there is an $\varepsilon>0$ such that any non-planar, non-Y polyhedral minimal cone has density greater than $3/2 + \varepsilon$.}. 
By an easy topological argument the flows $\cm^j$ have
to have triple junctions present in $C_{3/2}(0)$: 
consider the function 
$$h:\R^{n+k}\ra \R^{n-1}, (x_1,\ldots, x_{n+k})\mapsto
(x_2,\ldots,x_n)\ .$$ 
By Sard's Theorem and since the flows converge smoothly to $\mathcal{Y}$ away
from the triple edge of $\mathcal{Y}$, the preimage of every regular value of
$h$ in $M^j_t\cap B_{3/2}(0)$ is a smooth embedded curve in a
$(k+1)$-dimensional subspace $N =\{(x_2,\ldots,x_n) = \text{const}\}$, which in
the annulus
$B^N_{3/2}(0)\setminus B^N_\varepsilon(0)$ is close to three half-rays meeting at the origin under 120
degrees. But this already implies that there has to be a triple
junction present in $B_\varepsilon(0)$. 

{\bf Claim:} There exists $C>0$ and $N \in \N$ such that 
$$ \tilde{K}_{2,\alpha; C_1(O)}(\cm^j) \leq C$$
for $j>N$.

We closely follow the proof of \cite[Theorem 3.2]{White05}. Let us
assume that no such $C$ exists. Fix a $\delta>0$ to be chosen later.  As in \cite[Proposition 2.8]{White05} we see that for $\eta \searrow 0$ and for each $j$
$$ \tilde{K}_{2,\alpha; C_{1-\eta}((0, -\eta))}(\cm^j) \rightarrow \tilde{K}_{2,\alpha; C_{1}(O)}(\cm^j),$$
since $C_{1-\eta}((0, -\eta))$ is compactly contained in $C_1(O)$. Thus for a subsequence, relabelled the same, we can pick $\eta_j \searrow 0$ such that
$$ \tilde{K}_{2,\alpha; U_j}(\cm^j) = s_j < \infty\, ,\ \  s_j \rightarrow \infty\, ,
$$
where $U_j = C_{1-\eta_j}((0, -\eta_j))$. 
Choose $X_j$ on a triple edge of $\cm^j\cap U_j$ such that
\begin{equation}\label{eq:blowup.1}
d(X_j,U_j) \tilde{K}_{2,\alpha}(\cm^j,X_j)> \frac{1}{2} s_j\ .
\end{equation}
By translating, we may assume $X_j = O$. By dilating, we may assume that
\begin{equation*}
\tilde{K}_{2,\alpha}(\cm^j,O) = \delta\, .
\end{equation*}
By \eqref{eq:blowup.1} this implies
\begin{equation}\label{eq:blowup.3}
d(O,U_j) \rightarrow \infty\, .
\end{equation}
Now let $X$ be on a triple edge of $\cm^j\cap U_j$. Then
\begin{equation*}
d(X,U_j) \tilde{K}_{2,\alpha}(\cm^j,X) \leq s_j \leq 2 d(O,U_j) \tilde{K}_{2,\alpha}(\cm_j,O) = 2 \delta d(O,U_j)\, .
\end{equation*}
Thus 
\begin{equation*}\begin{split}
\tilde{K}_{2,\alpha}(\cm^j,X)\leq 2\delta \frac{d(O,U_j)}{d(X,U_j)} \leq 2 \delta \frac{d(O,U_j)}{d(O,U_j) - \|X\|}=
2\delta \Big(1-\frac{\|X\|}{d(0,U_j)}\Big)^{-1} 
\end{split}\end{equation*}
provided the right hand side is positive. By \eqref{eq:blowup.3}, this implies that
$\tilde{K}_{2,\alpha}(\cm^j,\cdot)$ is uniformly bounded by $4\delta$ on compact subsets
of space-time for $j$ large enough on the triple edges
of $\cm^j$.

Since the initial flows converge weakly to $\mathcal{Y}$, the flows
$\cm^j$ have Gaussian density ratios bounded by $3/2 +\varepsilon_j$, where
$\varepsilon_j\ra 0$, on increasingly large subsets
of space-time. Thus we can extract a subsequence converging to a limit
integral Brakke flow $\cm^\infty$, and on $C_1(O)$ the
convergence is in $C^2$ with one triple edge. Note that this flow
has Gaussian density ratios bounded by $3/2$ at all scales. But since
$\cm^\infty$ has a smooth triple edge at $O$ it has to be backwards
self-similar, 
and thus (by rotating) we may assume that it coincides with $\mathcal{Y}$. 

Again, the convergence is smooth away from the triple edge of $\mathcal{Y}$,
and 
\[
  \tilde{K}_{2,\alpha}(\cm^j,\cdot)\leq 4 \delta
\]
 along the triple edge of $\cm^j$ (on
compact regions of spacetime, for large $j$).
By choosing $\delta$ sufficiently small, we can apply Remark~\ref{rem:k2alphanorm} and Corollary~\ref{lem:maincor} to get 
$$\tilde{K}_{2,\alpha}(\cm^j, O) \ra 0,$$
and  we arrive at a contradiction, which proves the claim.

The above claim and the topological argument imply that the flows
$\cm^j\cap C_{3/2}(0)$ have exactly one triple edge, and that the edge is
$C^2$-close to the triple edge of $\mathcal{Y}$. Since
$\tilde{K}_{2,\alpha}(\cm^j,\cdot)$ is uniformly bounded along the triple
edge, and away from the triple edge the convergence is smooth, 
 we can again use Remark~\ref{rem:k2alphanorm} and Corollary~\ref{lem:maincor} to deduce that $\cm^j$
converges in $C^{2,\alpha}$ to $\mathcal{Y}$ on $C_1(O)$.

To see smooth convergence, one can replace the $C^{2,\alpha}$-norm by
any  $C^{k,\alpha}$-norm for $k\geq 3$ in the statement and proof
of Proposition \ref{mainprop} and the proof of Theorem \ref{theorem1}
to get analogous statements for all $k\geq 3$.
\end{proof}

\section{Extension to $Y$-regular  Brakke flows}\label{section-extension}

We consider  integral $n$-Brakke flows in $\R^{n+k}$ with the property
that every point of Gaussian density one is fully regular, i.e., has a
space-time neighbourhood in which the flow is smooth. In particular, we
assume that regular points cannot suddenly vanish. 
We call such flows
{\em unit regular}. 
We begin by showing that the class of unit-regular flows is 
closed under weak convergence
of Brakke flows:

\begin{lemma}\label{unit-lemma}
Suppose $\cm^i$ is a sequence of integral Brakke flows on the time interval $-\infty\le t\le 0$
such that $D_\lambda\cm^i=\cm^i$ for all $\lambda>0$.
Suppose the Gaussian density of $\cm^i$ at $O$ tends to $1$.
Then (for all sufficiently large $i$) $\cm^i$ is a non-moving plane.
\end{lemma}

\begin{proof}
By the Allard regularity theorem, there is an $\eps>0$ (depending on dimension)
such that the density of a stationary integral varifold at a singular point must be $\ge 1+\eps$.
In the lemma, we can assume that the Gaussian density $\Theta(\cm^i,O)<1+\eps$ for all $i$.
In the flow $\cm^i$, the surface at time $-1$ is a stationary integral varifold with
respect to a certain Riemannian metric, and the density is everywhere $\le \Theta(\cm^i,O)<1+\eps_i$.
Hence the flows $\cm^i$ are everywhere smooth for times $<0$.
The local regularity theory of \cite{White05} then gives uniform curvatures estimates
in $\{(t,x): -1<t<0, |x|<1\}$ and therefore in $\{(t,x): -1<t\le 0, |x|\le 1\}$.
Since $\cm^i$ is self-similar and smooth at $O$, it must in fact be planar.
Since $\cm^i$ is an integral Brakke flow with Gaussian density $<2$, in fact the Gaussian 
density must be $1$.
\end{proof}

\begin{theorem}\label{unit-theorem}
Suppose $\cm^i$ are unit-regular flows that converge weakly to a flow $\cm$.
Suppose $X$ is a point of $\cm$ at which the Gaussian density is $1$.
Then the $\cm^i$ converge smoothly to $\cm$ in a space-time neighbourhood of $X$
(and therefore $\cm$ is also unit regular.)
\end{theorem}

\begin{proof}
It suffices to show that if $X_i\in \cm^i$ converges to $O$,
then $X_i$ is regular for all sufficiently large $i$.
For in that case, there is a space-time neighbourhood $U$ of $O$
such that $\cm^i\cap U$ is smooth for all sufficiently large $i$.
By choosing $U$ small (and by monotonicity), we can assume that the Gaussian density
ratios of $\cm^i\cap U$ are $\le 1+\eps$ (for any specified $\eps>0$).
The local regularity theory in \cite{White05} then gives uniform $C^{k,\alpha}$ estimates
on the $\cm^i\cap U$ (for every $k$).

Thus suppose that $X_i\in\cm^i$ converges to $X$ and that $\Theta(\cm,X)=1$.
Let $\Theta_i=\Theta(\cm^i,X_i)$. 
By upper semicontinuity of the Gaussian density (which follows from Huisken's monotonicity),
$\Theta_i\to 1$.
Let $\mathcal{T}^i$ be a tangent flow to $\cm^i$ at $X_i$.
Then $\Theta(\mathcal{T}^i,O)=\Theta_i\to 1$, so (by Lemma \ref{unit-lemma}), 
  $\Theta_i=\Theta(\mathcal{T}^i,O)=1$ for all sufficiently
large $i$.  Since $\cm^i$ is unit regular, this implies that $\cm^i$ is regular at $X_i$.
\end{proof}

We say that  a triple junction point of
a flow is a space-time point at which at least one tangent flow is
a stationary union of 3 $n$-dimensional half-planes. For the next few paragraphs (until Theorem \ref{Y-regular-theorem}), 
singular points refer to those points at which the Gaussian density is $>1$.
In particular, we will regard  triple junction points, even well-behaved ones, as singular points.

Recall that the entropy of a flow $\cm$ is the supremum of the Gaussian density ratios at all points and scales. Fix a number $\zeta$ that is $>3/2$ and that is $<$ the Gaussian
density of a shrinking circle. Note that the density of a shrinking circle is less two. Let $C$ be the class of unit regular
$n$-dimensional Brakke flows in $\R^{n+k}$ with $-\infty < t \le 0$
and with entropy $\le \zeta$. This assumption restricts the type of possible tangent flows and thus implies estimates on the size of the singular set for each flow in $C$, see the proof of the following proposition.

Let $\cm$ be a flow in $C$.
We say that a smooth submanifold $D$ of $ \R^{n+k}$ intersects $\cm(t)$ transversely
provided $D$ is disjoint from the singular set of $\cm(t)$ and provided $D$ intersects
the regular set transversely.
  
\begin{proposition}\label{proposition-no triple}
Suppose that $\cm\in C$ has no triple junction points at a certain time~$t$.
 If $D$ is a flat $(k+1)$
dimensional disk such that $\partial D$
intersects $\cm(t)$ transversely, then it does so in an even number of points.
\end{proposition}

\begin{proof}
By the stratification theorem 9 in \cite{White97}, the set of singular points at time $t$
(i.e.~the set of points at time $t$ at which the Gaussian density is $>1$) has Hausdorff
dimension at most $(n-2)$. This follows since by assumption $\cm$ has no triple junction points at $t$ and the assumption on the entropy rules out tangent flows which have the form $\R^{n-1}\times S$, where $S$ is a one-dimensional self-similar shrinker, compare with Table 2 on p.~27 in \cite{White97}.

Let $D^\perp$ be the $(n-1)$-dimensional linear subspace of $\R^{n+k}$ perpendicular to $D$.
Let $Q$ be the projection of the singular set of $\cm(t)$ onto $D^\perp$.
Then $Q$ has Hausdorff dimension at most $(n-2)$, 
so almost every $v\in D^\perp$ lies in $Q^c$.  That is, for almost every $v\in D^\perp$, 
the disk $D+v$ is disjoint
from the singular set of $\cm(t)$.  It follows that for almost every $v\in D^\perp$, the disk 
$D+v$ intersects
$\cm(t)$ transversely.   
For such a $v$, by elementary topology, $(D+v)\cap \cm(t)$ is a finite disjoint union of compact curves,
and hence $\partial (D+v)\cap \cm(t)$, the set of endpoints of those curves, has an even number
of points.  Note that for small $v$, $\partial (D+v)\cap \cm(t)$ and $\partial D\cap \cm(t)$ have
the same number of points.
\end{proof}

\begin{remark}\label{open remark}
Proposition \ref{proposition-no triple} remains true (with essentially the same proof)
if $\cm(t)$ is allowed to contain  triple junction points, provided there is some open set containing $D$ 
in which $\cm(t)$ has no triple points.
(This is because, in the proof,
we only need $D+v$ to be disjoint from the singular set when $v$ is small.)
\end{remark}

\begin{proposition}\label{propsition triples}
Suppose $\cm\in C$.
Then the set of times at which the flow has a triple junction point is an open set.
Furthermore, if $X=(x,t)$ is a triple junction point of $\cm$ 
and if $\cm^i$ is a sequence of flows in $C$ converging to $\cm$, then (for sufficiently large $i$)
$\cm^i(t)$ has a triple junction point $x_i$ where $x_i\to x$.  
\end{proposition}

\begin{proof}
Let $(x,t)$ be a triple junction point of $\cm$.
Note that there exists a small flat $(k+1)$ dimensional disk $D$ centred at $x$ such that
\begin{equation}\label{little disk}
\text{$\partial D$ intersects $\cm(t)$ transversely in exactly $3$ points}.
\end{equation}
It follows that $D$ intersects $\cm(\tau)$ transversely in exactly $3$ points for all $\tau$
sufficiently close to $t$.  By Proposition \ref{proposition-no triple}, there are triple junction points at every such time $\tau$.
This proves openness.

Similarly, if $\cm^i$ converges to $\cm$, then for all sufficiently large $i$, 
$\partial D$ intersects $\cm^i(t)$ transversely in exactly $3$ points (by~\eqref{little disk}
and by smooth convergence, see Theorem \ref{unit-theorem}).  Hence (for such $i$) $\cm^i$ contains a triple
point $(x_i,t)$.
By Remark~\ref{open remark}, there must be a sequence
of such triple points $x_i$ whose distance to $D$ tends to $0$.   
Since $D$ can be arbitrarily small, the standard diagonal argument gives
a sequence $x_i$ converging to $x$.
\end{proof}

\begin{lemma}[First isolation lemma] \label{first isolation lemma}
Suppose $\cm^i$  is a sequence of flows in $C$, each with entropy $\le 3/2$, that
converges to a static $\mathcal{Y}$ flow $\cm$.  Then for all sufficiently large $i$, $\cm^i$
is a static $\mathcal{Y}$ flow.
\end{lemma}

\begin{proof}
By Proposition~\ref{propsition triples}, there are triple junction points $(x_i,0)$ in $\cm_i$ converging to $O$.
The result follows immediately by the equality case of monotonicity.
\end{proof}

\begin{corollary}\label{open-closed-corollary}
Let $Q$ be the set of $\cm\in C$ such that $\cm$ has entropy $\le 3/2$ and such that $O$ is a singular
point of $\cm$.  
Let $Q_Y$ be the subset of $Q$ consisting of non-moving configurations of three half-planes.
Then $Q_Y$ is an open and closed subset of $Q$.
\end{corollary}

\begin{proof}
It is clearly compact, hence closed.  Openness follows
from Lemma \ref{first isolation lemma}.
\end{proof}

\begin{theorem}\label{regularity-theorem}
Let $\cm^i$ be a sequence of dilation-invariant flows in $C$ that converge to a static $\mathcal{Y}$-flow. 
Then for all sufficiently large $i$, the only singularities of $\cm^i$  with $t<0$ are triple-junction
singularities.
\end{theorem}

\begin{proof}
Let $X_i=(x_i,t_i)$ be a singularity of $\cm^i$ at some time $t_i<0$.
By scaling, we can assume that $X_i\to O$.

Let $F_i$ be the closure of the set 
\[
  \{   \mathcal{D}_\lambda(\cm^i - X_i)^-: \lambda>0\},
  \]
where ${}^-$ is the restriction to past space-time $\R^{n+k}\times (-\infty,0)$. Note that if we fixed $\lambda=1$, the limit is the flow $\cm$.

Let $F$ be the set of all subsequential limits (as $i\to\infty$) of flows in $F_i$.

Then $F$ is a connected set of flows in $C$ such that
\begin{enumerate}
\item each flow in $F$ has entropy $\le 3/2$,
\item each flow has a singularity at $O$,
\item $F$ contains the flow $\mathcal{Y}$.
\end{enumerate}

By Corollary~\ref{open-closed-corollary}, $F$ consists only of static $Y$ configurations.

Note that $F_i$ contains all the tangent flows to $\cm^i$ at $X_i$.
Each such tangent flow is a tangent cone to $\cm^i(t_i)$ at $x_i$
times $(-\infty,0]$.

We have shown:
if $x_i$ is a singular point of $\cm_i(-1)$ and if $C_i$ is a tangent
cone to $\cm^i(-1)$ at $x_i$, then $C_i$ converges subsequentially
to a $Y$-cone.

By Corollary 2 in \cite{Simon93}, see also the remark after Theorem
7.3 in \cite{Simon93}, $C_i$ consists locally of three $C^{1,\alpha}$ sheets meeting
along a $C^{1,\alpha}$-edge. Note that the corollary discusses only tangent
cones, but the proof shows that the result is also true for a
stationary varifold $M \in\cm$ sufficiently weakly close to a $Y$-cone. 
Thus $C_i$ is a $Y$-cone for all sufficiently large $i$.
\end{proof}

Combining Corollary 2 in \cite{Simon93} as above with  \cite{KinderlehrerNirenbergSpruck78} (see also \cite{Krummel14} and for the higher codimension case \cite{Krummel15}) this implies:

\begin{corollary}\label{regularity-corollary}
For large $i$, $\cm^i(-1)$ consists locally of smooth $n$-manifolds that meet in threes
at equal angles along smooth $(n-1)$-manifolds.
\end{corollary}

In the following Lemma we do not require the flows $\cm^i$
to be in the class $C$ or unit regular.

\begin{lemma}[Second isolation lemma] \label{second-isolation-lemma}
 Let $\cm^i$ be a sequence of dilation invariant flows that converge
 to a flow $\cm$ that is a non-moving union of 3 half-planes. Then for
 all large enough $i$, $\cm^i$ is a  non-moving union of 3 half-planes.
\end{lemma}

\begin{proof}
  The monotonicity formula implies that for large enough $i$
  the flows have entropy less than $\zeta$. Since they are dilation
  invariant, Allard's regularity theorem implies that they are unit
  regular and thus belong to $C$. By Theorem \ref{regularity-theorem}
  and Corollary \ref{regularity-corollary} we can apply Theorem
  \ref{theorem1} to see that for large enough $i$ the flows have
  uniform smooth estimates up to time zero. But this implies that they
  are flat, smooth cones and equal to a non-moving union of 3 half-planes.
\end{proof}

We consider the class of $Y$-regular flows, i.e.,  unit-regular flows with the
additional property that at a triple junction point the flow has a
space-time neighbourhood in which it is smooth. From now on we will
consider such a point as a regular point.

\begin{theorem}\label{Y-regular-theorem}
  Let $\cm^i$ be a sequence of $Y$-regular flows that converge to
  $\mathcal{Y}$ in $C_2(O)$. Then for large enough $i$ the flows
  $\cm^i$ are smooth in $C_1(O)$.
\end{theorem}

\begin{proof}
We denote by $d$ the weak distance on Brakke flows induced by the
weak convergence of Brakke flows. Let $Q_Y$ be the space of all
non-moving unions of three half-planes in $\R^{n+k}, -\infty<t<0$ that are dilation invariant with respect
to the origin in space-time. By Corollary \ref{second-isolation-lemma}
there exists $\eps_0>0$ such that
$$d(\cm', Q_Y) \geq \eps_0$$
for any dilation invariant flow $\cm'$ in $\R^{n+k}, -\infty<t<0, \,
\cm'\not \in Q_Y$.\\[1ex]
Let $\eps=\eps_0/2$. Assume that $\cm^i$ is a sequence of $Y$-regular flows that converges
 to $\mathcal{Y}$ and that there are singular points $X_i
= (x_i,t_i)$ in $\cm^i\cap C_1(O)$. 
 By the monotonicity formula, we can choose $\eta_i \rightarrow \infty$ such
that the flows 
$$\tilde{\cm}^i := \cd_{\eta_i}(\cm^i-X_i)$$
satisfy
$$d(\tilde{\cm}^i, Q_Y) = \eps$$
and 
$$d(\cd_\lambda\tilde{\cm}^i,Q_Y)<\eps$$
for any $0<\lambda_i<\lambda<1$ and $\lambda_i \ra 0$. Note that the Gaussian density 
ratios of $\tilde{\cm}^i$  up to a radius $R_i$, on $C_{R_i}(O)$, where
$R_i\rightarrow \infty$, are bounded above by $3/2+\delta_i$ with $\delta_i\rightarrow
0$. Let $\bar{\cm}$ be a subsequential limit of $\tilde{\cm}^i$. Note that 
\begin{equation}\label{eq:epsilon-distance}
d(\bar{\cm}, Q_Y) = \eps
\end{equation}
and 
$$d(D_\lambda\bar{\cm},Q_Y)\leq\eps  $$
for any $0<\lambda<1$. Furthermore the entropy of $\bar{\cm}$ is
bounded from above by $3/2$.

Note first that this implies that any tangent flow at $-\infty$ of $\bar{\cm}$ is in $Q_Y$. 
Thus we see as in
Proposition~\ref{propsition triples} that the flow $\bar{\cm}$ has a
triple junction point $Z$, which implies that $\bar{\cm} - Z$ is backwards self-similar and thus $\bar{\cm} -Z \in Q_Y$.
Since the points $X_i$ are singular, the origin has to lie on
the edge of $\bar{\cm}$ and thus $\bar{\cm}$ is dilation invariant
with respect to $O$ as well. This yields a contradiction to \eqref{eq:epsilon-distance}.
\end{proof}

\begin{corollary}
  \label{Y-regular-corollary}
The class of $Y$-regular flows is closed under weak convergence.
\end{corollary}

\section{Short-time existence}\label{section-short-time}
We aim to show smooth short-time existence under Mean Curvature Flow
for initial smooth, compact surface clusters with smooth triple edges,
but no higher order junctions. We employ Ilmanen's elliptic regularisation
scheme \cite{Ilmanen94} to construct a Brakke flow starting at the
initial smooth surface cluster and use our
previous estimates to show that the flow remains smooth for short time. We
recall the construction of Ilmanen, adapted to our setting, and its properties
needed in the sequel.

\begin{theorem}[\cite{Ilmanen94}, section 8.1]\label{matching-motion}
  Let $T_0$ be local
integral $n$-current in $\R^{n+k}$ with $\partial T_0 = 0$ and
finite mass ${\bf M}[T_0]<\infty$. Then there exists a local integral
$(n+1)$-current $T$ in $\R^{n+k}\times [0,\infty)$ and a family
$\{\mu_t\}_{t\geq 0}$ of Radon measures on $\R^{n+k}$ such that
\begin{itemize}
\item[$(i)$] (a) $\partial T = T_0$,\\[1ex]
  (b) ${\bf M}[T_B]$, where $T_B = T\res (\R^{n+k} \times B),\ B\subset
  \R$, is absolutely continuous with respect to $\mathcal{L}^1(B)$.\\[-2ex]
\item[$(ii)$] (a) $\mu_0=\mu_{T_0}, {\bf M}[\mu_t]\leq {\bf M}[\mu_0]$
  for $t>0$.\\[1ex]
(b) $\{\mu_t\}_{t\geq 0}$ is an integral $n$-Brakke flow.\\[-2ex]
\item[$(iii)$] $\mu_t\geq \mu_{T_t}$ for each $t\geq 0$, where $T_t$ is
  the slice $\partial(T\res (\R^{n+k}\times [t,\infty))$.  
\end{itemize}
\end{theorem}

We outline the main steps of the proof. Ilmanen constructs local
integral $(n+1)$-currents $P^\eps$ in $\R^{n+k}\times \R$ that
minimise the
elliptic translator functional
$$ I^\eps[Q] = \frac{1}{\eps}\int e^{-z/\eps}\, d\mu_Q(x,z)\, ,$$ 
where $z$ is the coordinate in the additional $\R$-direction, subject
to the boundary condition
$$\partial Q = T_0\, ,$$
and $\R^{n+k}$ is identified with the height zero slice in
$\R^{n+k}\times \R$. Note that $I^\eps$ is the area functional for the
metric $\bar{g} = e^{-2z/((k+1)\eps)}(g \oplus dz^2)$, where $g \oplus
dz^2$ is the product metric on $\R^{n+k}\times \R$. 

The associated Euler-Lagrange equation implies that the family of
Radon measures $\mu^\eps_t=\mu_{P^\eps_t}$ corresponding to
$$P^\eps (t) = (\sigma_{-t/\eps})_\# (P^\eps)$$ 
for $0\leq t<\infty$, where
$\sigma_{-t/\eps}(x,z)=(x,z-t/\eps)$, is a downward translating 
integral $(n+1)$-Brakke flow on the relatively open subset
$W^\eps := \{(x,z,t)\, :\, z>-t/\eps,\ t\geq 0\}$ of space-time
$(\R^{n+k}\times \R)\times [0,\infty)$. 

Ilmanen's compactness theorem for Brakke flows implies that there is a
sequence $\eps_i\ra 0$ such that $\{\mu^{\eps_i}_t\}_{t\geq 0}$ converges
to a Brakke flow $\{\bar{\mu}_t\}_{t\geq 0}$ on
space-time. Furthermore, Ilmanen shows that $\bar{\mu}_0 = \mu_{T_0\times \R}$ and
$\bar{\mu}_t$ is invariant in the $z$-direction, which yields the
desired solution $\{\mu_t\}_{t\geq 0}$ via slicing. 

The integral current $T$ is constructed via considering a
subsequential limit of $T^{\eps_i}:=(\kappa_{\eps_i})_\#(P^{\eps_i})$ where
$\kappa_{\eps_i}(x,z)=(x,\eps_i z)$, which can be seen as an approximation
to the space-time track of $\{\mu_t\}_{t\geq 0}$ where now the
$z$-direction is considered as the time direction. Point $(iii)$
above verifies this interpretation.

We now consider $M_0$, a smooth, compact $n$-surface cluster in
$\R^{n+k}$, i.e.~a finite union of compact manifolds-with-boundary
that meet each other at 120 degree angles along their smooth
boundaries and no higher order junctions. We say that $M_0$ is
{\it orientable} if there exists an assignment of orientations to the
regular (non-triple junction) parts of $M_0$ in such a way that along
each edge, the three sheets that meet all induce the same orientation
on the edge.  

We will assume for the moment that $M_0$ is orientable (we will see
later that this is in fact not necessary). This orientation determines
a flat chain mod 3 whose support is $M$: given the
orientation, we give each piece multiplicity 1 to get the flat
chain. We again denote this flat chain with $T_0$. Note that $\partial
T_0 =0$. Ilmanen's construction works now analogously by replacing
local integral currents by flat chains mod 3, to obtain a flat 3-chain
$T$ and a Brakke flow $\{\mu_t\}_{t\geq 0}$ with the properties as in Theorem
\ref{matching-motion}. The existence of the minimisers $P^\eps$ follows from a fundamental compactness theorem for flat chains with coefficients in a group $G$. For a brief introduction to flat chains with coefficients in a group $G$, together with the corresponding references, see section 3 in \cite{White96}. 
\begin{lemma} \label{Y-regular-lemma}
  The approximating flows $\{\mu^\eps_t\}_{t\geq 0}$ are
  $Y$-regular for $t>0$. Thus the flow $\cm = \{\mu_t\}_{t\geq 0}$ is $Y$-regular for $t>0$.
\end{lemma}

\begin{proof}
  Since the approximating flows are translating solutions, it suffices
  to show that the flat 3-chains $P_\eps$ satisfy:
  \begin{enumerate}
  \item[(i)] $\mu_{P_\eps}$ is smooth in a neighbourhood of each point
    with density one.\\[-2ex]
\item[(ii)] $\mu_{P_\eps}$ is smooth in a neighbourhood of each point
  where a tangent cone is a static union of 3 unit density half-planes.
  \end{enumerate}
Recall that the $P_\eps$ are area minimising in the metric $\bar{g} =
e^{-2z/((k+1)\eps)}(g \oplus dz^2)$ on $\R^{n+k}\times
[0,\infty)$. Since flat 3-chains are equipped with the size norm, (i)
follows from Allard's regularity theorem. (ii) follows again from
Corollary 2 in \cite{Simon93} combined with  \cite{KinderlehrerNirenbergSpruck78} (see also \cite{Krummel14} and for the higher codimension case
\cite{Krummel15}). Direct regularity of minimising flat chains mod 3 in $\R^3$ was shown in \cite{Taylor73}.

Corollary
\ref{Y-regular-corollary} implies that the flow
$\{\bar{\mu}_t\}_{t\geq 0}$ is $Y$-regular for $t>0$. Since $\cm=\{\mu_t\}_{t\geq
  0}$ is obtained from $\{\bar{\mu}_t\}_{t\geq 0}$ by slicing, $\cm$
is also $Y$-regular for $t>0$.
\end{proof}

\begin{proposition}\label{initial-Y-lemma}
Let $\cm$ be a $Y$-regular flow in $0<t<\infty$ such that $\cm(t)$ converges
(as Radon measures) as $t\to 0$ to $M_0$, a regular cluster.
Then $\cm$ is smooth on some interval $0<t<T$.  Furthermore, the norm of the
second fundamental form at $(x,t)$ is $o(t^{-1/2})$.
\end{proposition}

\begin{proof}
Let $K(\cm,X)$ denote the largest principal curvature of $\cm$ at $X$ if $\cm$
is a unit-density point or a triple junction point of $\cm$. Otherwise, let $K(\cm,X)=\infty$.

Suppose the lemma is false.  Then there is a sequence $X_i=(x_i,t_i)$ in $\cm$ with $t_i\to 0$
such that
\[
     K(\cm,X_i)\,  |t_i|^{1/2} \to k \in (0,\infty].
\]
By hypothesis, we can extend the flow $\cm$ to $t=0$ by letting $\cm(0)$ be the Radon
measure associated to $M_0$.
Translate the flow $\cm$ by $-X_i$ and dilate parabolically by $1/\sqrt{t_i}$ to get
a flow $\cm^i$ defined on the time interval $-1\le t <\infty$.  Note that
\begin{equation}\label{eq:nonflat}
  K(\cm^i,O) \to k.
\end{equation}
By passing to a subsequence, we can assume that $\cm^i$ converges to a flow $\cm'$
that is $Y$-regular for $t>-1$.

Note that $\cm'(-1)$ is either a multiplicity $1$ affine plane, or the union of three multiplicity $1$
affine halfplanes meeting at $120^\circ$ angles.

If $\cm'(-1)$ is a multiplicity $1$ plane, then, by monotonicity, $\cm'(t)$ is equal to that plane
for all $t>0$.  But that implies that the $X_i$ are regular points of $\cm_i$, a contradiction.

Thus $\cm'(-1)$ is a union of three multiplicity $1$ affine half-planes meeting at equal angles.
We claim that $\cm'(t)=\cm'(-1)$ for all $t\ge -1$.
For if not, let $T$ be the infimum of times $t$ for $\cm'(t)\ne \cm'(-1)$.
Then $\cm'(T)=\cm'(-1)$.  (This could fail if sudden vanishing occurred. However,
$\cm'$ is $Y$-regular, and therefore there is no sudden vanishing.)

From the discussion above we see that
around any point $(y,T)$ in $\cm'(T)$, which is not a triple junction
point, there is a space-time neighbourhood in which $\cm'$ is
smooth. Thus there is $t>T$ and a $(k+1)$-dimensional disk $D$ such
that $\partial D$ intersects $\cm'(t)$ transversely in exactly 3
points. Since the entropy of $\cm'$ is bounded above $\zeta$ we can
argue as in Proposition \ref{propsition triples} to see that $\cm'$ has a triple point $X'=(x',t')$ at some time $t'>T$.
By monotonicity, $\cm' \cap \{-1\le t < t'\}$ is self-similar about $X'$.
It follows that $\cm'(t)=\cm'(-1)$ for $-1<t<t'$, contradicting the choice of $T$.

We have shown that $\cm'$ is a union of three half-planes for all $t\ge -1$.
By Theorem \ref{Y-regular-theorem}, the convergence $\cm^i\to\cm'$ is smooth for times $t>-1$, which implies that
\[
  K(\cm^i,O)\to 0,
\]
contradicting~\eqref{eq:nonflat}.
\end{proof}

\begin{remark}
In exactly the same way, one can show that $\tilde{K}_{2,\alpha}(\cm,(x,t))$
is $o(|t|^{-1/2})$, and likewise for the higher H\"older norms.
\end{remark}

\begin{proof}[Proof of Theorem \ref{short-time-existence-theorem}] In the case that $M_0$ is
  orientable, this follows from Proposition~\ref{initial-Y-lemma}, since
  the constructed flow $\cm$ is $Y$-regular for $t>0$.

Now assume that $M_0$ is not orientable. Note that  each point $p
\in M_0$ has a connected neighbourhood that is orientable, and there
are exactly two orientations. We call a choice of one of those two
orientations an orientation at $p$. Let $\tilde{M}_0(p)$ be the set of
the two orientations at $p$, and let $\tilde{M}_0 = \cup_{p\in
  M_0}\tilde{M_0}(p)$.

Note that $\tilde{M}_0$ is naturally a manifold where separate sheets
are meeting smoothly along triple junctions. Note also that
$\tilde{M}_0$ is orientable and comes with an orientation. Let $U$ be an
$\eps$-neighbourhood of $M_0$ in $\R^{n+k}$ that is homotopy equivalent
to $M_0$. Then there is a double cover
$$\Pi:\tilde{U} \ra U$$
of $U$ corresponding to the double cover $\tilde{M}_0$ of $M_0$. We
can think of $\tilde{M}_0$ embedded in $\tilde{U}$. 

Let $\tilde{T}_0$ be the mod 3 cycle given $\tilde{M}_0$ with its
orientation and let $\sigma: \tilde{U}\ra \tilde{U}$ be the involution
that switches the two sheets of $\tilde{U}$, i.e. $\sigma(p)$ is the
unique $q\neq p$ such that $\Pi(q)=\Pi(p)$. Note that $\sigma
:\tilde{M_0}\ra \tilde{M}_0$ is orientation-reversing, so that
$$\sigma_\#\tilde{T}_0 = - \tilde{T}_0\, .$$  
  
Ilmanen's construction in space-time $\tilde{U}\times
\R$, with the restriction that one works only with flat chains mod 3 $\tilde{T}$
satisfying
$$ \sigma_\#\tilde{T} =  - \tilde{T}\, ,$$
yields a Brakke flow in $\tilde{U}$ which (as varifolds) is
invariant under $\sigma$ and thus descends to a Brakke flow with
initial surface $M_0$. From Proposition \ref{initial-Y-lemma} it follows
again that this Brakke flow is smooth for short time. 

The convergence of $\cm(t) \ra M_0$ as $t\ra 0$ in $C^1$ follows from the proof of
Proposition \ref{initial-Y-lemma}. The smooth convergence away from the
edges of $M_0$ follows from Corollary \ref{thm:initial_smooth.2}.
\end{proof}

\begin{corollary}\label{maximal-existence-corollary}
 Assume that codimension $k=1$. Let $M_0$ be as before, and $(M_t)_{0\leq t <T}$ be the maximal smooth
  evolution of $M_0$ as constructed above. Assume that $T<\infty$. Then at $T$ either 
$$\lim_{t\ra T}\sup_{M_t}|A| = \infty$$
   or two triple junctions collide. 
\end{corollary}

\begin{proof}
  Note that a bound on the supremum of the second fundamental form
  along the surfaces in the cluster implies a bound on the curvature
  of the triple edges, since the sheets meet under a 120 degree
  condition.

Assume $\lim_{t\ra T}\sup_{M_t}|A| < \infty$ and no triple junctions
collide as $t\ra T$. We want to argue that all tangent flows of the constructed Brakke flow $\cm=\{\mu_t\}_{t\geq 0}$ at $T$ are either static unit density planes or static unions of 3 unit density halfplanes. 

Note that the strong maximum principle implies that $M_t$ is embedded for $0<t<T$. Since the second fundamental form is uniformly bounded, and no triple junctions can collide, we see that all tangent flows at $T$ are either static planes or static unions of 3 halfplanes for $t<0$ (i.e.~they are quasi-static). The strong maximum principle again implies that the tangent flows are unit density. Furthermore, by Ilmanen's construction there is no sudden vanishing and thus all tangent flows at $T$ are either static unit density planes or static unions of 3 unit density halfplanes for all $t$. Since $\cm$ is $Y$-regular this implies that $T$ was not maximal. 
\end{proof}

 \section{Short-time existence for general initial clusters}\label{section-general-initial-clusters}
In this section we show that it is possible to modify the approach in the previous section, using flat chains with coefficients in a finite group, to prove existence of a Brakke flow, starting from a general surface cluster. For immiscible fluids in codimension one, a similar idea was used in  \cite{White96}.

We consider a $n$-dimensional surface cluster $M_0 \subset \R^{n+k}$ as in the introduction, but only requiring that there exists a closed set $Z\subset M_0$ with $\ch^{n-1}(Z) =0$, such that all $l$-dimensional faces, for $l<n-1$, are contained in $Z$ and away from $Z$ the $n$-faces meet in threes along their common $n-1$-dimensional faces.

\begin{remark}
 We note that the conditions on $M_0$ allow for higher order junctions (contained in $Z$) and that along the $(n-1)$-dimensional edges the separate sheets do not necessarily meet under equal angles.
\end{remark}

We now consider the $(k-1)$-th homology group 
$$G = H_{k-1}(\R^{n+k} \setminus M_0, \mathbb{Z}_2)$$
and give every nonzero element in the group the norm $1$. We note that $G$ is finite since $M_0$ consists of a finite union of $n$-manifolds. We aim to make $M_0$ into a $n$-chain with coefficients in $G$. Since every element in $G$ has order 2, we don't have to assign orientations, just multiplicities. We do this as follows:

Consider a face $M^i_0$ as above. Pick a point $p \in M^i_0 \setminus \partial M^i_0$ and for $\varepsilon >0$ sufficiently small, consider the $(k-1)$-sphere
$$ C =\{p+\varepsilon v\, | \, |v|=1\text{ and } v \text{ is normal to } M^i_0 \text{ at }p \}\ .$$
Let $g_i$ be the element in $G$  corresponding to that sphere. That is the multiplicity we assign to $M^i_0$. This makes $M_0$ into a $n$-chain $[M_0]$ with multiplicities in $G$.

We have to check that $[M_0]$ is a cycle. It is easy to check that the sum of the multiplicities of the faces at each $(n-1)$-dimensional edge is $0$. Thus $\partial [M_0]$ is supported in $Z$. But a $(n-1)$ chain supported in a set with $\ch^{n-1}$-measure zero vanishes. Thus $[M_0]$ is a cycle. Let $\bar{\mu}$ be the radon measure associated to $[M_0]$, i.e.
$$\bar{\mu}(\varphi) = \sum_{i} \int_{M^i_0} \varphi |g_i|\, d\ch^{n} $$
for any $\varphi \in C^0_c(\R^{n+k})$.

\begin{theorem}\label{thm:existence_general}
Let $M_0$ be as above and $[M_0]$ the corresponding flat chain with multiplicities in $G$ as above. Then there exists a $Y$-regular $n$-dimensional Brakke flow $\{\mu_{t}\}_{t\geq 0}$ such that $\mu_0 = \bar{\mu}$, with the following property:\\[-4.2ex]
\begin{itemize}
 \item[(i)] If $x_0 \in M^i_0 \setminus \partial M^i_0$ for some $i$ and $|g_i| = 1$, then $\{\mu_t\}_{t\geq 0}$ is a smooth unit density mean curvature flow with triple edges in a space-time neighbourhood of $(x_0,0)$.
 \item[(ii)] If $x_0$ is on a $(n-1)$-dimensional edge of $M_0$, where in a neighbourhood of $x_0$ three sheets $M^{i_1}_0, M^{i_2}_0, M^{i_3}_0$ meet along a common  $(n-1)$-dimensional edge under equal angles and $|g_{i_1}|=|g_{i_2}|=|g_{i_3}|=1$, then $\{\mu_t\}_{t>0}$ is a smooth unit density mean curvature flow with triple edges in a space-time neighbourhood of $(x_0,0)$ and the initial cluster $M_0$ is locally attained in $C^1$.
\end{itemize}
\end{theorem}

\begin{remark}
 If a face $M^i_0$ is assigned zero multiplicity, it vanishes instantly under the flow. As an example consider $M_0 \subset \R^{n+1}$ consisting of two circles, with a line segment joining them. Each circle gets assigned a nonzero multiplicity, but the segment gets multiplicity 0. Therefore the segment vanishes instantly and we get two shrinking circles. 

\end{remark}

\begin{proof} We have shown that $[M_0]$ is a cycle. The existence proof follows completely analogously as in Section \ref{section-short-time}, but we work with flat chains with coefficients in $G$ instead of flat $3$-chains. For the existence of the minimisers $P^\eps$, see again \cite{White96}.

As in Lemma \ref{Y-regular-lemma} it follows that the approximating flows $\{\mu^\eps_t\}_{t\geq 0}$ are $Y$-regular for $t>0$, and thus $\{\mu_t\}_{t\geq 0}$ is $Y$-regular for $t>0$. The statements about the initial regularity of the flow follows again from Proposition \ref{initial-Y-lemma} and Corollary \ref{thm:initial_smooth.2}. 
\end{proof}

\section{Smooth short-time existence for non-regular initial networks}\label{section-nonregular-network}
In this chapter we consider mean curvature flow with triple edges for curves in arbitrary codimension. We call such a flow a {\it network flow}.

We call a network of curves  $N_0$ in $\R^{1+k}$ {\it non-regular}, if it has the form
$$N_0 = \cup_{i=1}^N \gamma^i_0\, ,$$
where the $\gamma^i_0$ are smooth, embedded curves, which are disjoint away from their endpoints. We assume that they meet in triples at their endpoints, but we do not require that the exterior unit normals add up to zero. Nevertheless we still assume that at each triple point the exterior unit normals are pairwise distinct, i.e.~the angle between any two curves meeting at a triple point is greater than zero. We call such a point a {\it non-regular} triple point. We aim to show that there exists a smooth network flow, starting from such a closed non-regular initial network.

\begin{lemma}\label{network-tangent-flow-lemma} Let $\mathcal{N}$ be a smooth self-similarly shrinking network flow on $\R^{1+k}\times(-\infty,0)$ with Gaussian density less than $2$, at most one triple junction and no closed loops. Then $\mathcal{N}$ is, up to a rotation, a static line, or a $\mathcal{Y}$.
\end{lemma}

\begin{proof}
 Since the Gaussian density of $\mathcal{N}$ is less than two, each branch $\gamma_i$ of $N_{-1}$ has multiplicity one and is embedded away from its endpoints. Each $\gamma_i$ is a smooth, embedded curve, satisfying
 $$\vec{k}= -\frac{x^\perp}{2}\, .$$
Since this is an ODE of second order, $\gamma_i$ is contained in a 2-dimensional plane through the origin and a member of the family of self-similarly shrinking solutions to curve shortening flow in the plane, classified by Abresch-Langer \cite{AbreschLanger86}. Since $N_{-1}$ contains no loops, $\gamma_i$ is diffeomorphic either to a line or a half-line. Since $\gamma_i$ is embedded, the classification of Abresch-Langer implies that $\gamma_i$ is contained in a line through the origin. In the case that $\mathcal{N}$ has no triple junctions, $N_{-1}$ can only be a line through the origin. In the other case it can only be the union of three half lines, meeting at equal angles at the origin.
\end{proof}

\begin{theorem}\label{short-time-existence-nonregular-network-theorem}
 Let $N_0$ be a smooth, compact, non-regular network in $\R^{1+k}$. Then there exists a $T>0$ and a smooth solution to the network flow $(N_t)_{0<t<T}$ such that $N_t \ra N_0$ in $C^0$. The convergence is in $C^1$ in a neighbourhood of each initial, regular triple junction and in $C^\infty$ away from all triple junctions.  Furthermore, the norm of the curvature is $O(t^{-1/2})$.
\end{theorem}

\begin{proof}
To simplify notation we assume that $N_0$ has only one non-regular triple junction at the origin. The general case follows then easily by performing the following approximation scheme simultaneously at each non-regular triple junction.

First note that there exist $r, \tau, \varepsilon>0$ with the following properties 
\begin{enumerate}
 \item $N_0 \cap B_{r}(0)$ consists of three curves, close to three half-lines, meeting at the origin.
 \item For all $(x,t)\in B_{r}(0)\times (0,\tau)$ it holds
 \begin{equation*}\label{eq:network.1}
 \int_{N_0} \rho_{x,t}(\cdot,0) \, d\mathcal{H}^1 \leq 2-\varepsilon\, .
 \end{equation*}
\item For all $(x,t) \in (B_{2r}(0)\setminus B_{r}(0))\times (0,\tau)$
 \begin{equation*}\label{eq:network.2}
 \int_{N_0} \rho_{x,t}(\cdot,0) \, d\mathcal{H}^1 \leq 3/2-\varepsilon\, ,
 \end{equation*}
\end{enumerate}
where $\rho_{x,t}(\cdot,0)$ is the backwards heat kernel in the monotonicity formula, centred at $(x,t)$.

For a sequence $s_i \ra 0$, $0<s_i<r$ we modify $N_0$ in $B_{s_i}(0)$ such that the three curves are meeting at equal angles to obtain a regular initial network $N_0^i$.  This can for example be done by glueing in, at a sufficiently small scale, a self-expanding network coming out of the tangent cone at $0$ of $N_0$. We can further assume that $(1)$ and $(2)$ continue to hold with $\varepsilon$ replaced by $\varepsilon/2$ and that $N^i_0 \ra N_0$ in $C^0$. 

As in the proof of Theorem \ref{short-time-existence-theorem} there exist $Y$-regular flows $\mathcal{N}_i$ starting at 
$N^i_0$ which are smooth on $0<t<T_i$. Note first that the proof of Proposition~\ref{initial-Y-lemma} implies that there is $T'>0$ such that $\mathcal{N}_i$ is smooth for $0<t<T'$ outside of $B_{r}(0)\times (0,\tau)$ for all $i$ with curvature bounded by $C/t^{1/2}$. Let $T:=\min\{\tau,T'\}$. The monotonicity formula together with $(1)$ and $(2)$ implies
\begin{enumerate}
 \item[(4)] The Gaussian density ratios of $\mathcal{N}_i$ are bounded above by $2-\varepsilon/2$ in $B_{r}(0)\times (0,T)$.\\[-1ex]
 \item[(5)] The Gaussian densities of $\mathcal{N}_i$ in $(B_{2r}(0)\setminus B_{r}(0))\times (0,T)$ are less than 3/2. Thus no triple junctions can cross the annulus, and $\mathcal{N}_i$ has exactly one triple junction in $B_{r}(0)\times (0,\min\{T_i,T\})$. 
\end{enumerate}

We claim that $\mathcal{N}^i$ is smooth in $B_{r}(0)\times (0,T)$ and has curvature bounded by $C/t^{1/2}$: Assume $T_i<T$. We first note that \cite[Lemma 8.5]{IlmanenNevesSchulze14} directly generalises to arbitrary codimension. Since $\mathcal{N}_i$ only has one triple junction in $B_{r}(0)\times (0,\min\{T_i,T\})$ and no closed loops, it implies that any tangent flow at $(x,T_i)$, where $x\in B_r(0)$, is a smooth, self-similarly shrinking network without closed loops and at most one triple junction. Since $\mathcal{N}_i$ is $Y$-regular, Lemma \ref{network-tangent-flow-lemma} together with $(4)$ implies that $\mathcal{N}_i$ is smooth in a neighbourhood of $(x,T_i)$. Thus $T_i\geq T$. 

One can similarly use Lemma \ref{network-tangent-flow-lemma} to check that the proof of \cite[Theorem 8.8]{IlmanenNevesSchulze14} directly generalises to show that the curvature is also bounded by $C/t^{1/2}$ in $B_{r}(0)\times (0,T)$.

Extracting a subsequential limit as $i \ra \infty$ one obtains a $Y$-regular flow $\mathcal{N}$ which is smooth for $0<t<T$ such that $N_t \rightarrow N_0$ as radon measures. The convergence in $C^\infty$ away from the triple junctions and in $C^1$ in a neighbourhood of the regular triple junctions follows as in the proof of Theorem \ref{short-time-existence-theorem}. The convergence in $C^0$ at the non-regular triple point at the origin follows from the approximation, since the bound on the curvature implies that the speed of the triple junction is also bounded by $C/t^{1/2}$.
\end{proof}

\begin{appendix}
\section{Initial regularity}
We verify that any local solution to Brakke flow, which can be written as the graph of a $C^1$-function and has smooth initial data, is actually locally smooth and attains its initial data smoothly. We work again with parabolic $C^{q,\alpha}$ spaces. 

\begin{theorem}\label{thm:initial_smooth.1} Let $U:=B_1(0)\subset \R^n$ and assume that 
$$ u:U\times[0,1) \rightarrow \R^k$$
is $C^{1,\alpha}$ and that $\text{graph}(u)$ defines a unit density Brakke flow. If $u(\cdot,0):U\rightarrow \R^k$ is smooth, then $u:U\times[0,b) \rightarrow \R^k$ is smooth.
\end{theorem}
\begin{proof}
 Let the nonparametric equation for mean curvature flow be given by
 $$\partial_t u = Q(u,Du,Du^2) .$$
 We extend $u$ to $U\times (-1,1)$ by setting
 $$u(x,t) = u(x,0) + t Q(u(x,0), Du(x,0), D^2u(x,0))$$
 for $t<0$. Note that since $u(\cdot, 0)$ is smooth, $u$ is smooth on $U\times(-1,0]$ and $C^{1,\alpha}$ on $U\times(-1,1)$. The graph of $u(\cdot, t)$ does not anymore constitute a Brakke flow for $t\in (-1,1)$, but it does so with transport term as follows: For $x \in \R^{n+k}$ write $x=(x',x'')$ where $x'\in \R^n, x'' \in \R^k$ and let $ f:U\times (-1,1) \rightarrow \R^k$ be defined by
 \begin{equation*}
f(x,t) = 
Q(u(x',0), Du(x',0), D^2u(x',0)) - Q(u(x',t), Du(x',t), D^2u(x',t))
\end{equation*}
for $t\leq 0$ and $f(x,t) =0$ otherwise. Note that $f$ is $C^{0,\alpha}$ on $U\times(-1,1)$. Then $\text{graph}(u)$ constitutes a Brakke flow with transport term as considered in \cite{Tonegawa14}. We can thus apply \cite[Theorem 6.3]{Tonegawa14} (with $g=0$) to see that $u$ is $C^{2,\alpha}$ on $U\times(-1,1)$. Therefore $u$ is $C^{2,\alpha}$ on $U\times [0,1)$. Now by standard parabolic PDE theory, $u$ is smooth on $U\times [0,1)$.
\end{proof}

\begin{remark} The same proof shows that if the assumption $u(\cdot,0)$ is smooth is relaxed to $u(\cdot,0)$ is in $C^{q, \alpha}$ for $q\geq 4$, then $u$ is $C^{q,\alpha}$ on $U\times[0,1)$. 
 \end{remark}
 
 \begin{corollary} \label{thm:initial_smooth.2} Let $U:=B_1(0)\subset \R^n$ and assume that 
$$ u:U\times[0,1) \rightarrow \R^k$$
is $C^{1}(U\times [0,1)) \cap C^\infty(U\times(0,1))$ and that $\text{graph}(u)$ defines a unit density Brakke flow. If $u(\cdot,0):U\rightarrow \R^k$ is smooth, then $u:U\times[0,b) \rightarrow \R^k$ is smooth.
\end{corollary}

\begin{proof} To apply Theorem \ref{thm:initial_smooth.1} we need to show that $u$ is $C^{1,\alpha}$ on $U\times[0,1)$. First note that for higher codimension mean curvature flow, balls of radius $R(t) = (R_0 - 2n t)$ still act as barriers. Thus we can adapt the final barrier argument in Section 9.11 in \cite{Edelen16} to show that $u \in C^{1,1}(U\times [0,1))$. 
\end{proof}

\begin{remark}It suffices to assume that $u\in C^1(U\times [0,1))$ since Brakke's local regularity theorem \cite{Brakke} (or alternatively \cite{Tonegawa14}) implies that $u$ is smooth for $t>0$. By the proof of Proposition \ref{initial-Y-lemma} this even extends to unit regular Brakke flows: the initial surface is attained locally in $C^1$ and thus the higher initial regularity extends. 
\end{remark}

\end{appendix}


\begin{thebibliography}{10}

\bibitem{AbreschLanger86}
Uwe~Abresch and Joel~Langer, \emph{The normalized curve shortening flow and
  homothetic solutions}, J. Differential Geom. \textbf{23} (1986), no.~2,
  175--196.

\bibitem{Brakke}
Kenneth Brakke, \emph{The motion of a surface by its mean curvature}, Princeton
  Univ. Press, 1978.

\bibitem{BronsardReitich93}
Lia Bronsard and Fernando Reitich, \emph{On three-phase boundary motion and the
  singular limit of a vector-valued {G}inzburg-{L}andau equation}, Arch.
  Rational Mech. Anal. \textbf{124} (1993), no.~4, 355--379.

\bibitem{DepnerGarckeKohsaka14}
Daniel Depner, Harald Garcke, and Yoshihito Kohsaka, \emph{Mean curvature flow
  with triple junctions in higher space dimensions}, Arch. Ration. Mech. Anal.
  \textbf{211} (2014), no.~1, 301--334.

\bibitem{Edelen16}
Nick Edelen, \emph{The free-boundary brakke flow}, 2016, {\tt
  arXiv:1602.03614}.

\bibitem{Freire10b}
Alexandre Freire, \emph{Mean curvature motion of graphs with constant contact
  angle at a free boundary}, Anal. PDE \textbf{3} (2010), no.~4, 359--407.

\bibitem{Freire10a}
\bysame, \emph{Mean curvature motion of triple junctions of graphs in two
  dimensions}, Comm. Partial Differential Equations \textbf{35} (2010), no.~2,
  302--327.

\bibitem{Ilmanen94}
Tom Ilmanen, \emph{Elliptic regularization and partial regularity for motion by
  mean curvature}, Mem. Amer. Math. Soc. \textbf{108} (1994), no.~520, x+90.

\bibitem{IlmanenNevesSchulze14}
Tom Ilmanen, Andr\'e Neves, and Felix Schulze, \emph{On short time existence
  for the planar network flow}, 2014, {\tt arXiv:1407.4756}.

\bibitem{KimTonegawa15}
Lami Kim and Yoshihiro Tonegawa, \emph{On the mean curvature flow of grain
  boundaries}, 2015, {\tt arXiv:1511.02572}.

\bibitem{KinderlehrerNirenbergSpruck78}
David~Kinderlehrer, Louis~Nirenberg, and Joel~Spruck, \emph{Regularity in elliptic free
  boundary problems}, J. Analyse Math. \textbf{34} (1978), 86--119 (1979).

\bibitem{Krummel15}
Brian Krummel, \emph{Regularity of minimal submanifolds and mean curvature
  flows with a common free boundary}, in preparation.

\bibitem{Krummel14}
\bysame, \emph{Regularity of minimal hypersurfaces with a common free
  boundary}, Calc. Var. Partial Differential Equations \textbf{51} (2014),
  no.~3-4, 525--537.

\bibitem{MantegazzaNetworks}
Carlo Mantegazza, Matteo Novaga, and Vincenzo~Maria Tortorelli, \emph{Motion by
  curvature of planar networks}, Ann. Sc. Norm. Super. Pisa Cl. Sci. (5)
  \textbf{3} (2004), no.~2, 235--324.

\bibitem{Simon93}
Leon Simon, \emph{Cylindrical tangent cones and the singular set of minimal
  submanifolds}, J. Differential Geom. \textbf{38} (1993), no.~3, 585--652.

\bibitem{Simon97}
\bysame, \emph{Schauder estimates by scaling}, Calc. Var. Partial Differential
  Equations \textbf{5} (1997), no.~5, 391--407.

\bibitem{Taylor73}
Jean~E. Taylor, \emph{Regularity of the singular sets of two-dimensional
  area-minimizing flat chains modulo {$3$} in {$R\sp{3}$}}, Invent. Math.
  \textbf{22} (1973), 119--159.

\bibitem{Tonegawa14}
Yoshihiro Tonegawa, \emph{A second derivative {H}\"older estimate for weak mean
  curvature flow}, Adv. Calc. Var. \textbf{7} (2014), no.~1, 91--138.

\bibitem{TonegawaWickramasekera15}
Yoshihiro Tonegawa and Neshan Wickramasekera, \emph{The blow up method for
  brakke flows: networks near triple junctions}, 2015, {\tt arXiv:1504.01212}.

\bibitem{White96}
Brian White, \emph{Existence of least-energy configurations of immiscible
  fluids}, J. Geom. Anal. \textbf{6} (1996), no.~1, 151--161.

\bibitem{White97}
\bysame, \emph{Stratification of minimal surfaces, mean curvature flows, and
  harmonic maps}, J. Reine Angew. Math. \textbf{488} (1997), 1--35.

\bibitem{White05}
\bysame, \emph{A local regularity theorem for mean curvature flow}, Ann. of
  Math. (2) \textbf{161} (2005), no.~3, 1487--1519.

\end{thebibliography}

\providecommand{\bysame}{\leavevmode\hbox to3em{\hrulefill}\thinspace}
\providecommand{\MR}{\relax\ifhmode\unskip\space\fi MR }
\providecommand{\MRhref}[2]{%
  \href{http://www.ams.org/mathscinet-getitem?mr=#1}{#2}
}
\providecommand{\href}[2]{#2}

\end{document}